\documentclass[12pt,reqno]{amsart}
\usepackage[latin1]{inputenc}
\usepackage{mathptmx}
\usepackage{amscd}
\newtheorem{theorem}{Theorem}
\newtheorem{lemma}[theorem]{Lemma}

\newtheorem{proposition}[theorem]{Proposition}

\theoremstyle{definition}
\newtheorem{remark}[theorem]{Remark}
\newtheorem{definition}[theorem]{Definition}

\numberwithin{equation}{section}

\def\cN{{\mathcal N}}
\def\cP{{\mathcal P}}
\def\cZ{{\mathcal Z}}
\def\cT{{\mathcal T}}
\def\cC{{\mathcal C}}
\def\cH{{\mathcal H}}
\def\cF{{\mathcal F}}

\def\simp{{\textup{simp}}}
\def\nsmp{{\textup{nonsimp}}}

\def\ZZ{{\mathbb Z}}
\def\RR{{\mathbb R}}
\def\QQ{{\mathbb Q}}

\def\rank{\operatorname{rank}}
\def\inte{\operatorname{int}}

\def\para{\operatorname{par}}
\def\U{\operatorname{U}}
\def\gp{\operatorname{gp}}
\def\Ker{\operatorname{Ker}}
\def\Hilb{\operatorname{Hilb}}

\def\tdeg{\operatorname{tdeg}}

\def\bt{{\textbf{\textit t}}}

\let\dirsum=\oplus

\usepackage{pstricks}
\def\vertex{\pscircle[fillstyle=solid,fillcolor=black]{0.07}}
\definecolor{light}{gray}{0.9}
\definecolor{medium}{gray}{0.8}

\let\epsilon=\varepsilon
\let\tilde=\widetilde

\begin{document}
\title[Normaliz]{Normaliz: algorithms for affine monoids\\ and rational cones}

\author{Winfried Bruns and Bogdan Ichim}
\address{Universit\"at Osnabr\"uck, FB Mathematik/Informatik, 49069
Osnabr\"uck, Germany} \email{wbruns@uos.de, csoeger@uos.de}

\address{Institute of Mathematics, C.P. 1-764,
70700 Bucharest, Romania} \email{bogdan.ichim@imar.ro}

\begin{abstract}
Normaliz is a program for the computation of Hilbert bases of
rational cones and the normalizations of affine monoids. It may also
be used for solving diophantine linear systems. In this paper we
present the algorithms implemented in the program.
\end{abstract}

\maketitle

\section{Introduction}\label{intro}
\let\thefootnote\relax\footnotetext{Acknowledgement: The second author was partially supported by CNCSIS grant RP-1 no. 7/01.07.2009 during the preparation of this work.}

The program Normaliz  got its name from the first task for
which it was designed: the computation of normalizations of
affine monoids (or semigroups in other terminology).  This task
amounts to the computation of the Hilbert basis of the monoid
of lattice points in a rational cone $C$ with given generating
system $x_1,\dots,x_n$ (see for Section \ref{affmon} for
terminology and \cite{BG} for mathematical background). Such
cones can be described equivalently by homogeneous linear
diophantine equations and inequalities, and the computation of
the normalization is equivalent to solving such systems.

The mathematical aspects of the first implementation of Normaliz
have been documented in \cite{BK}. In this paper we present the
algorithms that have been added or modified in version 2.0 and
later. Further extensions, for example parallelization of time
critical steps, are still experimental; they will be presented in
\cite{BHIKS}.

As any other program that computes Hilbert bases, Normaliz
first determines a system of generators of the monoid. Section
\ref{reduction} describes Normaliz' approach for the reduction
of the system of generators to a Hilbert basis---often (but not
always) the most time consuming part of the computation.
Section \ref{dual} contains our implementation of the
Fourier--Motzkin elimination, which is tuned for obtaining best
results in the case when most of the facets are simplicial.
(Fourier--Motzkin elimination computes the convex hull of a
finite set of points, or, in homogenized form, the support
hyperplanes of a finitely generated cone.) We need this variant
for the new algorithm by which $h$-vector and Hilbert
polynomial are computed. It is based on line shellings and will
be presented in Section \ref{hviash}. Finally, our
implementation of Pottier's algorithm \cite{Pot} is presented
in Section \ref{dual_alg}. In our interpretation, this ``dual''
algorithm is based on a representation of the cone as an
intersection of halfspaces, whereas the ``primal'' algorithm of
Normaliz starts from a system of generators.

The first version of Normaliz was a C program created by
Winfried Bruns and Robert Koch in 1997--1998 and extended in
2003 by Witold Jarnicki. Version 2.0 (2007--2008) was
completely rewritten in C++ by Bogdan Ichim. Pottier's
algorithm for solving systems of inequalities and equations was
added in version 2.1. Christof Söger enhanced the user
interface in version 2.2, the currently public version. The
distribution of Normaliz \cite{Nmz} contains a Singular library
and a Macaulay2 package; for the latter, written by Gesa Kämpf,
see \cite{BKae}. Andreas Paffenholz provided a polymake
interface to Normaliz \cite{JMP}.

We wish to thank all colleagues who have contributed to the
development of of Normaliz.

\section{Affine monoids and their Hilbert bases}\label{affmon}

We use the terminology introduced as in \cite{BG}, but for the
convenience of the reader we recall some important notions. A
\emph{rational cone} $C\subset \RR^d$ is the intersection of
finitely many linear halfspaces
$H_\lambda^+=\{x\in\RR^d:\lambda(x)\ge0\}$ where $\lambda$ is a
linear form with rational coefficients (with respect to the standard
basis of $\RR^d$). By the theorem of Minkowski--Weyl (for example,
see \cite[1.15]{BG}), we can require equivalently that $C$ is of
type $\RR_+x_1+\dots+\RR_+x_n$ with $x_i\in\QQ^d$, $i=1,\dots,n$. In
this case, $x_1,\ldots,x_n$ form a \emph{system of generators} for
$C$. If $C$ can be generated by a linearly independent set of
generators, we say that $C$ is a \emph{simplicial} cone. If $\dim
C=d$, then the halfspaces in an irredundant representation of $C$ as
an intersection of halfspaces are uniquely determined, and the
corresponding linear forms $\lambda_i$ are called \emph{support
forms} of $C$, after they have been further specialized such that
$\lambda_i(\gp(M))=\ZZ$. If $\gp(M)=\ZZ^d$, the last condition
amounts to the requirement that the $\lambda_i$ have coprime
integral coefficients. (Such linear forms are called
\emph{primitive}.) In the following all cones are rational, and we
omit this attribute accordingly. A cone is \emph{pointed} if
$x,-x\in C$ implies $x=0$.

An \emph{affine monoid} $M$ is finitely generated and (isomorphic
to) a submonoid of a lattice $\ZZ^d$. By $\gp(M)$ we denote the
subgroup generated by $M$, and by $\rank M$ its rank. The support
forms $\sigma_1,\dots,\sigma_s$ of the cone $\RR_+M\subset \RR M$
are called the \emph{support forms of $M$}. They define the
\emph{standard map}
$$
\sigma:M\to\ZZ_+^s,\qquad
\sigma(x)=\bigl(\sigma_1(x),\dots,\sigma_s(x)\bigr).
$$
We introduce the \emph{total degree} $\tdeg x$ by $\tdeg
x=\sigma_1(x)+\dots+\sigma_s(x)$. (In \cite{BG} the total degree is
denoted $\tau$.)

The \emph{unit group} $\U(M)$ consists of the elements $x\in M$ for
which $-x\in M$ as well. It is not hard to see that $x\in\U(M)$ if
and only if $\sigma(x)=0$ (see \cite[2.14]{BG}), in other words, if
and only if $\tdeg x=0$. (However, in general $\tdeg x=\tdeg y$ does
\emph{not} imply $x-y\in\U(M)$ since $x-y$ need not belong to $M$.)
One calls $M$ \emph{positive} if $\U(M)=0$.

An element $x\in M$ is \emph{irreducible} if $x\notin \U(M)$ and a
representation $x=y+z$ with $y,z\in M$ is only possible with
$y\in\U(M)$ or $z\in U(M)$.

In the next definition we extend the terminology of \cite{BG}
slightly.

\begin{definition}
Let $M$ be a (not necessarily positive) affine monoid. A subset $H$
of $M$ is a \emph{system of generators modulo $\U(M)$} if
$M=\ZZ_+H+\U(M)$, and $H$ is a \emph{Hilbert basis} if it is minimal
with respect to this property.
\end{definition}

A Hilbert basis is necessarily finite since $M$ has a finite system
of generators. Moreover, every system of generators modulo $\U(M)$
contains a Hilbert basis. Often we will use the following criterion
(see \cite[2.14]{BG}).

\begin{proposition}\label{HilbCrit}
$H\subset M$ is a Hilbert basis if and only if it is a system of
representatives of the nonzero residue classes of the irreducible
elements modulo $\U(M)$.
\end{proposition}

The Hilbert basis of a positive affine monoid is uniquely determined
and denoted by $\Hilb(M)$.

Suppose $N$ is an overmonoid of $M$. Then we call $y\in N$
\emph{integral} over $M$ if $ky\in M$ for some $k\in\ZZ$, $k>0$. The
set of elements of $N$ that are integral over $M$ form the
\emph{integral closure} $\widehat M_N$ of $M$ in $N$; it is itself a
monoid. The \emph{normalization} $\bar M$ of $M$ is its integral
closure in $\gp(M)$, and if $M=\bar M$, $M$ is called \emph{normal}.

If $M$ is normal, the case in which we are mainly interested, then
$M$ splits in the form $\U(M)\dirsum\sigma(M)$ (see \cite[2.26]{BG})
and we can state:

\begin{proposition}\label{HilbNorm}
Let $M$ be a normal affine monoid with standard map
$\sigma:M\to\ZZ_+^s$. Then $H\subset M$ is a Hilbert basis of $M$ if
and only if $\sigma$ maps $H$ bijectively onto a Hilbert basis of
$\sigma(M)$.
\end{proposition}

It is a crucial fact that integral closures of affine monoids have a
geometric description (see \cite[2.22]{BG}):

\begin{theorem}
Let $M\subset N$ be submonoids of $\QQ^d$, and $C=\RR_+M$.
\begin{enumerate}
\item Then $\widehat M_N=C\cap N$.
\item If $M$ and $N$ are affine monoids, then $\widehat M_N$ is
affine, too.
\end{enumerate}
\end{theorem}

The second statement of the theorem is (an extended version of)
\emph{Gordan's lemma}.

The program Normaliz computes Hilbert bases of monoids of type
$C\cap L$ where $C$ is a pointed rational cone specified either (i)
by a system $x_1,\dots,x_n\in \ZZ^d$ or (ii) a system
$\sigma_1,\dots,\sigma_s\in(\RR^d)^*$ of integral linear forms, and
$L$ is a lattice that can be chosen to be either $\ZZ^d$ or, in case
(i), $\ZZ x_1+\dots+\ZZ x_n$. We will simply say that Normaliz
computes Hilbert bases of rational cones. If $C$ is pointed and
their is no ambiguity about the lattice $L$, then we simply write
$\Hilb(C)$ for $\Hilb(C\cap L)$.

Once a system of generators of $C$ is known (either from the input
data or as a result of a previous computation), Normaliz reduces
this computation to the \emph{full-dimensional} case in which $\dim
C=\rank L$ and, and introduces coordinates for the identification
$L=\ZZ^{\rank L}$. The necessary coordinate transformations are
discussed in \cite[Section 2]{BK}.

\section{Reduction}\label{reduction}

All algorithms that compute Hilbert bases of rational cones cannot
avoid to first produce a system of generators that is nonminimal in
general. In a second, perhaps intertwined, step the system of
generators is shrunk to a Hilbert basis. This approach is based on
the the following proposition. Let us say that $y\in M$
\emph{reduces} $x\in M$ if $y\notin \U(M)$, $x\neq y$, and $x-y\in
M$.

\begin{proposition}
Let $M$ be an affine monoid (not necessarily positive or normal),
$E\subset M$ a system of generators modulo $\U(M)$, and $x\in E$. If
$x$ is reduced by some $y\in E$, then $E\setminus\{x\}$ is again a
system of generators modulo $\U(M)$.
\end{proposition}

\begin{proof}
Note that $E$ contains a Hilbert basis. It is enough to show that
$E\setminus\{x\}$ contains a Hilbert basis as well. If $x-y\notin
\U(M)$, then $x$ is reducible, and does not belong to any Hilbert
basis, and if $x-y\in\U(M)$, we can replace $x$ by $y$ in any
Hilbert basis $H\subset E$ to which $x$ belongs.
\end{proof}

The proposition shows that one obtains a Hilbert basis from a set $E$ of
generators modulo $\U(M)$ by (i) removing all units from $E$, and
(ii) successively discarding elements $x$ such that $x-y\in M$ for
some $y\in E$, $x\neq y$. After finitely many reduction steps one
has reached a Hilbert basis.

The difficult question is of course to decide whether $x\in\U(M)$ or
$x-y\in M$. However, if $M=C\cap L$ with a rational cone
$C\subset\RR^d$, and a sublattice $L$ of $\QQ^d$, then this question
is very easy to decide, once the support forms
$\sigma_1,\dots,\sigma_s$ of $C$ are known:
$$
x-y\in M\quad\iff\quad x-y\in C \quad\iff\quad \sigma_i(x-y)\ge 0,\
i=1,\dots,s,
$$
and
$$
x\in\U(M)\quad\iff\quad \sigma_i(x)=0,\ i=1,\dots,s.
$$
Therefore, if $C$ is given by a set of generators, the necessity of
reduction forces us to compute the support forms of $C$. Normaliz's
approach to this task is discussed in Section~\ref{dual}.

It is very important for efficiency to make reduction as fast as
possible. Normaliz uses the following algorithm. The elements
forming a system of generators are inserted into a set $E$ ordered
by increasing total degree such that at the end of the production
phase $E$ contains at most one element from each residue class
modulo $\U(M)$ (and no element from $\U(M)$). Let
$E=\{x_1,\dots,x_m\}$. Then a Hilbert basis $H$ is extracted from
$E$. Initially, $H$ is the set of elements of minimal total degree in
$E$, say $H=\{y_1,\dots,y_u\}=\{x_1,\dots,x_u\}$. For
$i=u+1,\dots,m$ the element $x_i$ is compared to the dynamically
extended and reordered list $H=\{y_1\dots,y_n\}$ as follows:
\begin{itemize}
\item[(R1)]
for $j=1,\dots,n$,
\begin{itemize}
\item[(a)] if $\tdeg x_i<2\tdeg y_j$, then $x_i$ is appended
to $H$ as $y_{n+1}$;
\item[(b)] if $x_i-y_j\in M$, then $H$ is reordered as
$H=\{y_j,y_1,\dots,y_{j-1},y_{j+1},\dots,y_n\}$;
\end{itemize}
\item[(R2)] if (i) or (ii) does not apply for any $j$,
then $x_i$ is appended to $H$ as
$y_{n+1}$.
\end{itemize}
For the justification of this procedure, note that $x-y\in M$ for
some $y$ with $2\tdeg y\le \tdeg x$ if $x$ is reducible. Therefore
(R1)(a) can be applied, provided $ 2\tdeg y_k > \tdeg x_i$ for all
$k\ge j$. This holds since $E$ is ordered by ascending degree, the
fact that no element in $H$ that follows $y_j$ has been touched by
the rearrangement in (R1)(b): only elements  with $2\tdeg y_j \le
\tdeg x_i$ have been moved, and $\tdeg x_i\le \tdeg x_k$ for all
$k\ge i$.

\begin{remark}
(a) The ``darwinistic'' rearrangement in (R1)(b) above has a
considerable effect as all tests have shown. It keeps the
``successful reducers'' at the head of the list. Moreover,
successive $x_i$ are often close to each other (based on empirical evidence), so that $y_j$ has a
good chance to reduce $x_{i+1}$ if it reduces $x_i$.

(b) Instead of $\tdeg$ one can use any other convenient positive
linear form $\tau:M \mapsto \RR_+$ with the property that
$\tau(y)=0\Leftrightarrow y\in U(M)$.
\end{remark}

In Section \ref{dual_alg} we will encounter a situation in which the
subset $E$ to which reduction is to be applied need not to be a
system of generators. Then we call a subset $E'$ an
\emph{auto-reduction} of $E$, if $E'\cap \U(M)=\emptyset$ and no
element of $E'$ is reduced by another one.

\section{Computing the dual cone}\label{dual}

Let $C\subset\RR^d$ be a rational cone and $L\subset \ZZ^d$ a
lattice. In order to perform the reduction of a system of generators
of the normal monoid $C\cap L$ to a Hilbert basis, as discussed in
the previous section, one must know the hyperplanes that cut out $C$
from $\RR^d$, or rather the integral linear forms defining them.
These linear forms generate the dual cone $C^*$ in $(\RR^d)^*$.

Conversely, if $C$ is defined as the intersection of half-spaces
represented by a system of generators of $C^*$, then the use of an
algorithm based on a system of generators of $C$ makes it necessary
to find such a system. Since $C=C^{**}$ (see \cite[1.16]{BG}), this
amounts again to the computation of the dual cone: the passage from
$C$ to $C^*$ and that from $C^*$ to $C$ can be performed by the same
algorithm.

In the following we take the viewpoint that a full-dimensional
pointed rational cone $C\subset\RR^d$ is given by a system of
generators $E$, and that the linear forms generating $C^*$ are to be
computed. Normaliz uses the well-known Fourier-Motzkin elimination
for this task, however with a simplicial refinement that we will
describe in detail.

Fourier--Motzkin elimination is an inductive algorithm. It
starts from the zero cone, and then inserts the generators
$x_1,\dots,x_n$ successively, transforming the support
hyperplanes of $C'=\RR_+x_1+\dots+\RR_+x_{n-1}$ into those of
$C=\RR_+x_1+\dots+\RR_+x_{n}$. The transformation is given by
the following theorem; for example, see \cite[pp.~11,12]{BG}.

\begin{theorem}\label{Fourier}
Let $C$ be generated by $x_1,\dots,x_n$ and suppose that
$C'=\RR_+x_1+\dots+\RR_+x_{n-1}$ is cut out by linear forms
$\lambda_1,\dots,\lambda_m$. Let $\cP=\{\lambda_i:
\lambda_i(x_n)>0\}$, $\cN=\{\lambda_i: \lambda_i(x_n)<0\}$, and
$\cZ=\{\lambda_i: \lambda_i(x_n)=0\}$. Then $C$ is cut out by the
linear forms in the set
$$
\cP\cup\cZ\cup\{\lambda_i(x_n)\lambda_j-\lambda_j(x_n)\lambda_i:
\lambda_i\in\cP,\lambda_j\in\cN\}.
$$
\end{theorem}

In this raw form the algorithm produces $|\cP|\cdot|\cN|$
linear forms, from which the \emph{new} facets have to be
selected. While the complexity of this algorithm may seem
negligible in view of the subsequent steps in the Hilbert basis
computation, this is no longer so if applied in the computation
of a shelling of $C$ (see Section \ref{hviash}). But in the
computation of a shelling the boundary of (a lifting of) $C$
consists mainly of simplicial facets, and this allows an
enormous acceleration. (The construction of the lifting ensures
that most of its facets are simplicial; see Remark
\ref{FMsimp}.)

In a geometric interpretation of Fourier-Motzkin elimination, we
have to find the boundary $V$ of that part of the surface of $C'$
that is visible from $x_n$, or rather the decomposition of $V$ into
subfacets (faces of $C$ of dimension $d-2$).
\begin{figure}[bht]
$$
\psset{unit=1.5cm, dotsep=0.5pt}
\def\vertex{\pscircle[fillstyle=solid,fillcolor=black]{0.04}}
\begin{pspicture}(0,0)(4,3)
\def\A{0,0}
\def\B{0.5,3}
\def\C{1.2,1.7}
\def\D{2.5,1.3}
\def\E{4,1.6}
\pspolygon[fillstyle=solid, fillcolor=medium](\A)(\B)(\C)
\pspolygon[fillstyle=solid, fillcolor=light](\A)(\C)(\B)(\D)
\psline(\C)(\D)
\psline[linestyle=dotted](\A)(\E)
\psline[linestyle=dotted](\B)(\E)
\psline[linestyle=dotted](\C)(\E)
\psline[linestyle=dotted](\D)(\E)
\rput(\A){\rput(-0.2,0){$x_1$}}
\rput(\C){\rput(0.2,0.2){$x_2$}}
\rput(\B){\rput(-0.2,0){$x_3$}}
\rput(\D){\rput(0.2,0.2){$x_4$}}
\rput(\E){\rput(0.2,0){$x_5$}}
\rput(\A){\vertex}
\rput(\B){\vertex}
\rput(\C){\vertex}
\rput(\D){\vertex}
\rput(\E){\vertex}
\end{pspicture}
$$
\caption{Cross-section of extension of the cone}\label{extend}
\end{figure}
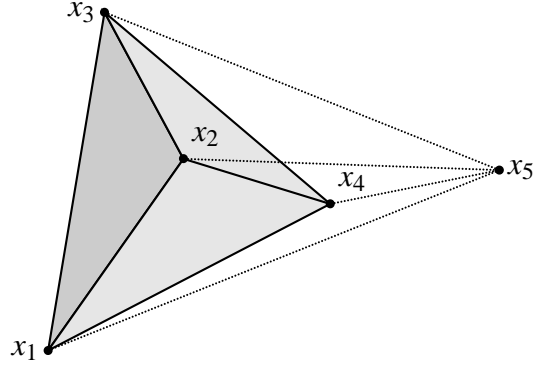
Figure \ref{extend} illustrates the inductive step of
Fourier--Motzkin elimination in the three-dimensional
cross-section of a four-dimensional cone. The area of the
``old'' cone visible from the ``new'' generator $x_5$  is the
union of the cones spanned by the triangles $[x_1,x_4,x_2]$ and
$[x_2,x_4,x_3]$, whereas $V$ is the union of the cones over the
line segments forming the cycle $[x_1,x_4,x_3,x_2,x_1]$.

Each subfacet $S$ of $C'$ is the intersection of two facets $F$ and
$G$ and we call $F$ and $G$ \emph{partners} with respect to $S$. In
order to compute the new facets of $C$ we have to find those
subfacets $S$ of $C'$ whose two overfacets belong to $\cP$ and
$\cN$, respectively. The new facets of $C$ are then the cones
$\RR_+(S\cup\{x_n\})$.

Let $E'$ be the subset of $E\setminus\{x_n\}$ whose elements are
contained in a hyperplane belonging to $\cP$ as well as in a
hyperplane belonging to $\cN$. Clearly, a facet $F$ of $C'$ can only
contribute to a new facet of $C$ if $|F\cap E'|\ge d-2$. While this
observation is useful (and is applied), its effect is often rather
limited.

Normaliz proceeds as follows; for simplicity we will identify
subsets of $E$ with the faces they generate.

(D1) It separates the facets in $\cP$ and $\cN$ into the subsets
$\cP_\simp$ and $\cN_\simp$ of simplicial ones and the subsets
$\cP_\nsmp$ and $\cN_\nsmp$ of nonsimplicial ones, discarding those
facets that do not satisfy the condition $|F\cap E'|\ge d-2$.

(D2) All subfacets of all the facets $N\in\cN_\simp$ are formed by
simply taking the subsets $S$ of cardinality $d-2$ of $N\cap E$
(which has cardinality $d-1$ in the simplicial case). The pairs
$(S,N)$ are stored in a set ordered by lexicographic comparison of
the components $S$. In fact, if $S$ appears with a second facet
$N'\in\cN_\simp$, then it cannot belong to $V$, and both pairs
$(S,N)$ and $(S,N')$ can be discarded immediately. Forming the
ordered set $\cT$ is of complexity of order $q\log_2 q$ where
$q=(d-1)|\cN_\simp|$.

(D3) Each pair $(S,N)\in \cT$ is compared to the facets in $G\in
\cN_\nsmp\cup\cZ$: if $S\subset G$, then the partner of $F$ with
respect to $S$ does not belong to $\cP$, and $(S,N)$ can be deleted
from $\cT$. (In the critical situation arising from the computation
of a shelling, the sets $\cN_\nsmp$ and $\cZ$ are usually short.)

(D4) At this point $\cT$ contains only pairs $(S,N)$ such that the
partner of $N$ with respect to $S$ indeed belongs to $\cP$, and
therefore gives rise to new facet. It remains to find the partners.

(D5) Normaliz now produces all subfacets $S$ of the facets
$P\in\cP_\simp$ and tries to find $S$ as the first component of an
element in the set $\cT$. This search is of complexity of order
$(d-1)\cdot|\cP_\simp|\cdot\log_2 q$, $q$ as above.

If the search is successful, a new facet of $C$ is produced, and the
pair $(S,N)$ is discarded from $\cT$.

(D6) To find the partners in $\cP_\nsmp$ for the remaining pairs
($S,N)$ in $\cT$, the sets $S$ are compared to the facets $\cP$ in
$\cP_\nsmp$. This comparison is successful in exactly one case,
leading to a new facet.

(D7) Finally, the facets $N\in\cN_\nsmp$ are paired with all facets
$P\in\cP$, as described in Theorem \ref{Fourier}, and whether a
hyperplane $H$  produced is really a new facet of $C$ is decided by
the following rules, applied in the order given:
\begin{enumerate}
\item[(i)] if $|N\cap P\cap E|<d-2$, then $H$ can be discarded;

\item[(ii)] if $|N\cap P\cap E|=d-2$ and $P\in\cP_\simp$, then $H$ is a new
facet;

\item[(iii)] $H$ is a new facet if and only if $\rank(N\cap P\cap E)=d-2$;

\item[(iv)] (alternative to the rank test) $H$ is a new facet if and only
if the only non-simplicial facets containing $N\cap P$ are $N$ and
$P$.
\end{enumerate}

Which of the tests (iii) or (iv) is applied, is determined as
follows: if the number of nonsimplicial facets is $<d^3$, then (iv)
is applied, and otherwise the rank test is selected.

It is not hard to see that (iv) is sufficient and necessary for
$N\cap P$ to have dimension $d-2$. Indeed, a subfacet is contained
in exactly two facets, and if we have arrived at step (iv), $P\cap
N$ cannot be contained in any simplicial facet $G$: since $P\cap
N\ge d-2$, it must be a subfacet contained in $G$, and it would
follow that $P=G$ or $N=G$, but both $P$ and $N$ are nonsimplicial.

\begin{remark}
(a) Computing the dual cone is essentially equivalent to computing
the convex hull of a finite set of points: instead of the affine
inhomogeneous system of inequalities we have to deal with its
homogenization. Therefore one could consider other convex hull
algorithms, like ``gift wrapping'' or ``beneath and beyond'' (see
\cite{J} for their comparison to Fourier--Motzkin elimination). The
main advantages of Fourier--Motzkin elimination for Normaliz are
that it does not require (but allows) the simultaneous computation
of a triangulation, and furthermore that the incremental
construction of $C$, adding one generator at a time, can be used
very efficiently in some hard computations (see \cite{BHIKS}).

(b) One can extend the idea of the simplicial refinement and work
with a triangulation of the boundary of $C'$ that is then extended
to a triangulation of the boundary of $C$, accepting that a facet
may decompose in many simplicial cones. In this way the pairing of
``positive'' and ``negative'' facets can be reduced to the creation
of a totally ordered set and the search in such a set. In our tests
the separate treatment of simplicial and nonsimplicial facets turned
out superior.
\end{remark}

\section{The primal Normaliz algorithm}\label{primal}

The \emph{primal} algorithm of Normaliz proceeds as follows (after
the initial coordinate transformation discussed in \cite[Section
2]{BK}), starting from a system of generators $x_1,\dots,x_n$ of
$C$:
\begin{itemize}
\item[(N1)] the support hyperplanes of $C$ are computed as described by
Fourier-Motzkin elimination (Section \ref{dual});
\item[(N2)] intertwined with (N1), the lexicographic (or placing)
triangulation of $C$ is computed into which the generators
$x_1,\dots,x_n$ are inserted in this order;
\item[(N3)] for each simplicial cone $D$ in the triangulation
$\Hilb(D\cap \ZZ^d)$ is determined;
\item[(N4)] the union of the sets $\Hilb(D\cap \ZZ^d)$ is reduced to
$\Hilb(C)$.
\end{itemize}
After the completion of (N1) one knows $C^*$ and can decide whether
$C$ is pointed since pointedness of $C$ is equivalent to
full-dimensionality of $C^*$ (see \cite[1.19]{BG}).

In step (N2) the lexicographic triangulation $\Sigma'$ of
$C'=\RR_+x_1,\dots+\RR_+x_{n-1}$ is extended to a lexicographic
triangulation $\Sigma$ of $C=\RR_+x_1,\dots+\RR_+x_{n}$ as follows:
Let $F_1,\dots,F_v$ be those facets of the maximal cones in
$\Sigma'$ that lie in the facets of $C$ visible from $x_n$: then
$\Sigma=\Sigma'\cup\{F_i+\RR_+x_n:i=1,\dots,v\}$. (If $x_n\in C'$,
then $C=C'$ and $\Sigma=\Sigma'$.) (Compare \cite[p.~267]{BG} for
lexicographic triangulations.) This construction is illustrated by
Figure \ref{triext}.
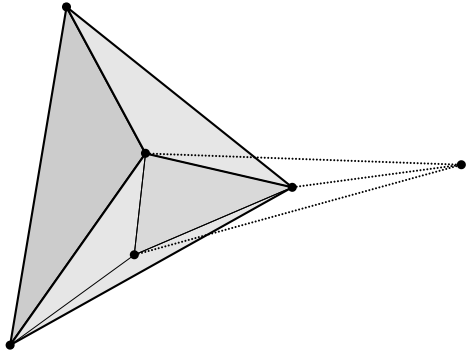
\begin{figure}[bht]
$$
\psset{unit=1.5cm, dotsep=0.5pt}
\def\vertex{\pscircle[fillstyle=solid,fillcolor=black]{0.04}}
\definecolor{veryverylight}{gray}{0.85}
\begin{pspicture}(0,0)(4,3)
\def\A{0,0}
\def\B{0.5,3}
\def\C{1.2,1.7}
\def\D{2.5,1.4}
\def\E{4,1.6}
\def\F{1.1,0.8}
\pspolygon[fillstyle=solid, fillcolor=medium](\A)(\B)(\C)
\pspolygon[fillstyle=solid, fillcolor=light](\A)(\C)(\B)(\D)
\pspolygon[linewidth=0pt,fillstyle=solid, fillcolor=veryverylight](\F)(\C)(\D)
\psline(\C)(\D)
\psline[linestyle=dotted](\C)(\E)
\psline[linestyle=dotted](\D)(\E)
\psline[linestyle=dotted](\F)(\E)
\psline[linewidth=0.3pt](\A)(\F)
\psline[linewidth=0.3pt](\C)(\F)
\psline[linewidth=0.3pt](\D)(\F)
\rput(\A){\vertex}
\rput(\B){\vertex}
\rput(\C){\vertex}
\rput(\D){\vertex}
\rput(\E){\vertex}
\rput(\F){\vertex}
\end{pspicture}
$$
\caption{Extension of triangulation}\label{triext}
\end{figure}

It only remains to explain how a set $E_D$ of generators of $D\cap
\ZZ^d$ is determined if $D$ is simplicial, i.e., generated by a
linearly independent set $V=\{v_1,\dots,v_d\}\subset \ZZ^d$.
Following the notation of \cite{BG}, we let
$$
\para(v_1,\dots,v_d)=\{a_1v_1+\dots+a_dv_d:0\le a_i<1,\
i=1,\dots,d\}
$$
denote the semi-open parallelotope spanned by $v_1,\dots,v_d$. Then
the set
\begin{equation}
E=E_D=\para(v_1,\dots,v_d)\cap\ZZ^d\label{defE}
\end{equation}
generates $D\cap \ZZ^d$ as a free module over the free submonoid
$\ZZ_+v_1+\dots+\ZZ_+v_d$. In other words, every element $z\in
D\cap\ZZ^d$ has a unique representation
\begin{equation}\label{freefree}
z=x+\sum_{i=1}^d a_iv_i,\qquad x\in E,\ a_i\in\ZZ_+.
\end{equation}
See \cite[2.43]{BG} for this simple, but crucial fact.   Clearly
$E\cup \{v_1,\dots,v_d\}$ generates the monoid $D\cap \ZZ^d$. Figure
\ref{FigSemiOpen} illustrates the construction of $E$.
\begin{figure}[htb]
$$
\psset{unit=0.7cm}
\definecolor{verylight}{gray}{0.75}
\definecolor{veryverylight}{gray}{0.85}
\newpsstyle{fyp}{fillstyle=solid,fillcolor=veryverylight}
\begin{pspicture}(0,-0.0)(5,6)
\pspolygon[style=fyp,linecolor=white,linewidth=0pt](0,0)(1.833,5.5)(4.5,5.5)(4.5,2.25)
\pspolygon[fillstyle=solid, fillcolor=verylight,linecolor=white,
linewidth=0pt](0,0)(1,3)(3,4)(2,1)
\psline{->}(0,0)(1,3)
\psline{->}(0,0)(2,1)
\multirput(0,0)(1,0){5}{\multirput(0,0)(0,1){6}{\vertex}}
\rput(0,0){\pscircle{0.18}}
\rput(1,1){\pscircle{0.18}}
\rput(1,2){\pscircle{0.18}}
\rput(2,2){\pscircle{0.18}}
\rput(2,3){\pscircle{0.18}}
\rput(-0.6,0){0}
\rput(2,0.6){$v_1$}
\rput(0.6,3){$v_2$}
\end{pspicture}
$$
\caption{The semi-open parallelotope and the set $E$}
\label{FigSemiOpen}
\end{figure}
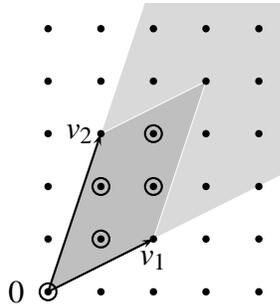

The efficient computation of $E_d$ has been discussed in \cite{BK};
it amounts to finding a representative $z$ for each residue class in
$\ZZ^d/(\sum \ZZ v_i)$ and reducing it modulo $v_1,\dots,v_d$ to its
representative in $\para(v_1,\dots,v_d)$.

In the following the attribute \emph{local} refers to the simplicial
cones $D$ whereas \emph{global} refers to $C$.

While the primal algorithm had been realized already in the first
version of Normaliz (see \cite{BK}), it has now undergone several
refinements.

\begin{remark}\label{multipl}
(a) The lexicographic triangulation is used by Normaliz in its
``normal'' (meaning ``standard'') computation type. It is replaced
by a shelling if the $h$-vector is to be computed (see Section
\ref{hviash}).

(b) The total number of vectors generated by Normaliz is the sum of
the multiplicities $\mu(D)=|\det(v_1,\dots,v_d)|$ of the simplicial
cones $D$ in the triangulation. If the monoid is defined by a
lattice polytope $P$, then this number is the $\ZZ^d$-normalized
volume of $P$, and therefore independent of the triangulation. (This
count includes the zero vector in each simplicial cone; therefore
the number of simplicial cones should be subtracted from the sum of
multiplicities.)
\end{remark}

\section{$h$-vectors via shellings}\label{hviash}

For $N\subset \RR^d$ we set
$$
H_N(\bt)=\sum_{x\in N\cap\ZZ^d}\bt^x.
$$
Here we use multi-exponent notation: $\bt^x=t_1^{x_1}\cdots
t_d^{x_d}$. The formal Laurent series is simply the ``characteristic
series'' of the set $N\cap\ZZ^d\subset\ZZ^d$. If $N$ is an
``algebraic'' object (for example, an affine monoid), then we can
interpret $H_N(\bt)$ as the multigraded Hilbert series of $N$.

Let $C$ be a cone and $M=C\cap\ZZ^d$. Suppose that $C$ is
triangulated by the conical complex $\cC$, the standard situation in
the primal algorithm of Normaliz. If $C_1,\dots,C_m$ are the maximal
cones in $\cC$, then
\begin{equation}\label{indH}
H_M(\bt)=\sum_{i=1}^m H_{D_i}(\bt),\qquad D_i=C_i\setminus
(C_1\cup\dots\cup C_{i-1}),\ i=1,\dots, m.
\end{equation}

Within $C_i$, the set $D_i$, $i=1,\dots,m$, is the complement of the
union of the sets of faces of $C_1,\dots,C_ {i-1}$. Therefore it is
the union of the (relative) interiors of those faces of $C_i$ that
are not contained in $C_1\cup\dots\cup C_{i-1}$. In order to compute
$H_{D_i}(\bt)$, one has to solve two problems: (i) to compute
$H_{\inte D}(\bt)$ for a simplicial cone $D$, and (ii) to find the
decomposition of $D_i$ as a union of interiors of faces of $C_i$.

As in Section \ref{primal} we denote the linearly independent
generators of the simplicial cone $D$ by $v_1,\dots,v_d$ and
consider the system of generators $E=E_D$. For a subset $Y$ of
$V=\{v_1,\dots,v_d\}$, let
$$
\cH_Y(\bt)=H_{\ZZ_+Y}(\bt)=\prod_{v_i\in
Y}\frac1{1-\bt^{y_i}}\quad\text{and}\quad\bt^Y=\prod_{v_i\in
Y}\bt^{v_i}.
$$
By definition, $\cH_Y(\bt)$ is the Hilbert series of the free monoid
generated by $Y$. In view of equation \eqref{freefree} one obtains
$$
H_D(\bt)=\cH_V(\bt)\sum_{x\in E} \bt^x.
$$

Now problem (i) is easily solved (compare \cite[p.~234]{BG}):
$$
H_{\inte(D)}=(-1)^{\dim D}H_D(\bt^{-1})=\cH_V(\bt)\sum_{x\in
E}\bt^{v_1+\dots+v_d-x}.
$$
Problem (ii) is very hard for an arbitrary order of the cones $C_i$
in the triangulation. However, it becomes easy if $C_1,\dots,C_m$ is
a shelling. Shellings are the classical tool for the investigation
of $h$-vectors, as demonstrated by McMullen's proof of the upper
bound theorem (see \cite{BH} or \cite{Zie}). We need the notion of
shelling only for complexes of simplicial cones (or polytopes), for
which it reduces to a purely combinatorial condition.

\begin{definition}
Let $\cC$ be a complex of simplicial cones (or polytopes) whose
maximal cones have constant dimension $d$. An order $C_1,\dots,C_m$
of the maximal cones in $\cC$ is called a \emph{shelling} if
$C_i\cap(C_1\cup\dots\cup C_{i-1})$ is a union of facets of $C_i$
for all $i$.
\end{definition}

The next lemma solves problem (ii) for a shelling. For a compact
formula we need one more piece of notation: for $x\in E$,
$x=a_1v_1+\dots+a_dv_d$, let
$$
[x]=\{v_i:a_i\neq 0\}.
$$

\begin{lemma}\label{hvectshell}
Let $D\subset\RR^d$ be a simplicial cone of dimension $d$ generated
by the linearly independent set $V=\{v_1,\dots,v_d\}\subset \ZZ^d$.
Let $G$ be the union of some facets $F$ of $D$, and set
$W=\bigcup_{F\subset G} V\setminus F$. Then
\begin{align*}
H_{D\setminus G}(\bt)&=\cH_V(\bt)\sum_{x\in
E} \bt^{-x}\bt^{W\cup[x]}\\
&=\cH_V(\bt)\sum_{x\in E} \bt^x \bt^{W\setminus[x]}.
\end{align*}
\end{lemma}

\begin{proof}
Let $Y\subset V=\{v_1,\dots,v_d\}$. For simplicity of notation we
set
$$
E(Y)=\para(Y)\cap\ZZ^d=\{x\in E:[x]\subset Y\},
$$
Moreover, note that since $D$ is simplicial, a face of $D$ is not
contained in $G$ if and only if it contains $W$.

Then
\begin{align*}
H_{D\setminus G}&=\sum_{Y\supset W} H_{\inte(\RR_+Y)}(\bt)
= \sum _{Y\supset W} \sum_{x\in E(Y)}
\bt^Y\bt^{-x}\cH_Y(\bt)
=\sum_{x\in E}\bt^{-x}\sum_{Y\supset W\cup[x]}\bt^Y \cH_Y(\bt)\\
&=\sum_{x\in
E}\bt^{-x}\bt^{W\cup[x]}\cH_{W\cup[x]}(\bt)\sum_{Z\subset
V\setminus(W\cup[x])}\bt^Z\cH_Z(\bt)\\
&=\cH_V(\bt)\sum_{x\in E}\bt^{-x}\bt^{W\cup[x]}.
\end{align*}

The proof of the second formula is actually simpler. Let $L$ be the
free monoid generated by $v_1,\dots,v_d$. Then
$D\cap\ZZ^d=\sum_{x\in E} x+L$. Now one computes $(x+L)\setminus G$,
and obtains the result.
\end{proof}

In the present implementation Normaliz uses the first formula in
Lemma \ref{hvectshell}, but only in the case in which there is an
integral linear form $\gamma$ such that the given generators of the
cone $C$ have value $1$ under $\gamma$ (this case is called
\emph{homogeneous}). Then $\gamma$ induces a $\ZZ$-grading on
$M=C\cap\ZZ^d$ in which all generators of all the simplicial cones
$C_1,\dots,C_m$ in the triangulation have degree $1$, and Lemma
\ref{hvectshell} specializes to
\begin{equation}
H_{C_i\setminus G}(t)=\frac1{(1-t)^d}\sum_{x\in E}
t^{|W\cup[x]|-\deg x} =\frac1{(1-t)^d}\sum_{x\in E}
t^{|W\setminus[x]|+ \deg x}.
\end{equation}
Therefore one needs only to count each element $x\in E$ (including
$0$!) in the right degree to obtain the $h$-vector of the cone $C$.

The price to be paid for the simple computation of the $h$-vector is
the construction of a shelling. The classical tool for this purpose
is a line shelling as introduced by Brugesser and Mani.

First we ``lift'' the cone $C\subset\RR^d$ generated by
$v_1,\dots,v_m$ to a cone $C'\subset \RR^{d+1}$ by extending the
generating elements by positive weights:
$$
v_i'=(v_i,w_i)\in\ZZ^{d+1},\quad w_i>0.
$$
The \emph{bottom} $B$ of $C'$ is the conical complex formed by all
the facets (and their faces) that are ``visible from below'', more
precisely by all the facets $F$ of $F'$ whose corresponding support
form $\sigma_F\in (\RR^{d+1})^*$ has positive last coordinate. The
projection $\RR^{d+1}\to\RR^d$, $(a_1,\dots,a_{d+1})\mapsto
(a_1,\dots,a_{d})$, maps $B$ bijectively onto $C$, and the images of
the facets constitute a conical subdivision of $C$. We always choose
the weights in such a way that the facets in the bottom of $C'$ are
simplicial, and therefore we obtain a triangulation of $C$. (This is
the classical construction of regular triangulations; compare
\cite[1.F]{BG}.)

It follows from \cite[Theorem 8.1]{Zie} that this triangulation is
shellable, and in order to reduce our conical situation to the
polytopal one in \cite{Zie}, one simply works with a suitable
polytopal cross-section of $C'$.

\begin{remark}\label{FMsimp}
Although it is superfluous, we also keep the ``top'' of $C'$
simplicial by a suitable choice of weights. The only facets of
$C'$ that cannot always be made simplicial are ``vertical''
ones, namely those parallel to the direction of projection.
Each vertical facet of $C'$ corresponds to a (non-simplicial)
facet of $C$ whereas the bottom and top facets correspond to
the simplicial cones in triangulations of $C$. Since such
triangulations usually have many more cones than $C$ has
support hyperplanes, $C'$ has mainly simplicial facets, and for
this reason we have developed the simplicial refinement of
Fourier--Motzkin elimination in Section \ref{dual}.
\end{remark}

The proof of \cite[Theorem 8.1]{Zie} tells us how to find a
shelling. We choose a point $x\in\inte(C')$ such that the ray
$x+\RR_+v$, $v=(0,\dots,0,-1)\in\RR^{d+1}$, is intersected at
pairwise different points $x+t_Fv$ by the linear subspaces $\RR F$
where $F$ runs through the facets in the bottom. Then we order the
facets by ascending ``transition times'' $t_F$. The images of the
facets $F$, ordered in the same way, yield a simplicial shelling of
$C$ since the projection preserves the face relation in the complex.
The construction of the shelling is illustrated by Figure
\ref{FigShell}.
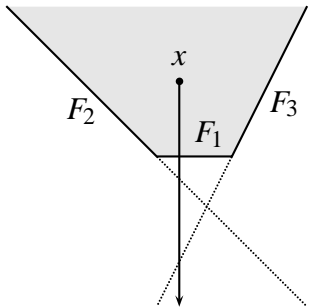
\begin{figure}
$$
\psset{unit=1.0cm, dotsep=0.5pt}
\def\vertex{\pscircle[fillstyle=solid,fillcolor=black]{0.05}}
\begin{pspicture}(0,-2)(4,2)
\pspolygon[linewidth=0pt,linecolor=white,fillstyle=solid, fillcolor=light]%
     (0,2)(2,0)(3,0)(4,2)
\psline(0,2)(2,0)(3,0)(4,2)
\psline[linestyle=dotted](2,0)(4,-2)
\psline[linestyle=dotted](3,0)(2,-2)
\psline{->}(2.3,1)(2.3,-2)
\rput(2.3,1){\vertex}
\rput(2.3,1.3){$x$}
\rput(2.7,0.3){$F_1$}
\rput(1.0,0.6){$F_2$}
\rput(3.7,0.7){$F_3$}
\end{pspicture}
$$
\caption{The line shelling}\label{FigShell}
\end{figure}

It is not difficult to produce a point $x$ in $\inte(C')$, but one
may need several attempts to ensure that the transition times are
all different. Instead we choose $x$ only once and then replace it
by a point infinitely near to $x$. This trick is known as "simulation of simplicity"
in computational geometry (see \cite{E}).

For the next lemma it is convenient to replace the integral support
forms $\sigma_F$ of the bottom faces by their rational multiples
$\rho_F=-\sigma_F/\sigma_F(v)$, $v=(0,\dots,0,-1)\in\RR^{d+1}$ as
above. These are normed in such a way that $\rho_F(v)=-1$ (and
$\rho_F/\sigma_F>0$).

\begin{lemma}\label{shell}
Let the bottom facets of $C'$ be ordered by the following rule:
$F<\tilde F$ if $\rho_F(x)<\rho_{\tilde F}(x)$ or
$\rho_F(x)=\rho_{\tilde F}(x)$ and $\rho_F$ precedes $\rho_{\tilde
F}$ in the lexicographic order on $(\RR^{d+1})^*$.

Then the bottom facets of $C'$ form a shelling in this order.
\end{lemma}

\begin{proof}
Note that there exists a weight vector $w\in\RR^{d+1}$ such that
$\rho_F$ precedes $\rho_{\tilde F}$ in the lexicographic order if
and only $\rho_F(w)<\rho_{\tilde F}(w)$. For sufficiently small
$\epsilon>0$ our ordering is identical with that obtained from the
inequality $\rho_F(x+\epsilon w)<\rho_{\tilde F}(x+\epsilon w)$.

The transition time $t_F$ of the ray $(x+\epsilon w)+\RR_+v$ with
the linear subspace spanned by $F$ is given by
$$
t_F=-\frac{\rho_F(x+\epsilon w)}{\rho_F(v)}=\rho_F(x+\epsilon w),
$$
and we have indeed ordered the facets by increasing transition
times.
\end{proof}

\begin{remark}
After the mathematical foundation  for the computation of Hilbert
functions has been laid in Lemmas \ref{hvectshell} and \ref{shell},
we describe the essential details of the implementation.

(S1) Normaliz computes the support hyperplanes of $C$---these are
needed anyway---and extracts the extreme integral generators from
the given set of generators in order to use the smallest possible
system of generators for $C'$.

(S2) The support hyperplanes of $C'$ are computed by Fourier-Motzkin
elimination with simplicial refinement as described in Section
\ref{dual}. It is here where the simplicial refinement shows its
efficiency since the bottom (and top) facets of $C'$ are kept
simplicial by a suitable ``dynamic'' choice of the weights.

Note that the vertical facets of $C'$, namely those parallel to $v$,
cannot be influenced by the choice of weights. They are determined
by the facet structure of $C$.

(S3) Once the support hyperplanes of $C'$ have been computed, the
bottom facets are ordered as described in Lemma \ref{shell}. Let
$C_1<\dots<C_m$ be the correspondingly ordered simplicial cones that
triangulate $C$. In order to apply Lemma \ref{hvectshell} we have to
find the intersections $C_i\cap (C_1\cup\dots\cup C_{i-1})$. To this
end we do the following: we start with an empty set $\cF$, and in
step $i$ ($i=1,\ldots,m$) we insert the facets of $C_i$ into $\cF$.
(i) If a facet is already in $\cF$, then it is contained in
$C_1\cup\dots\cup C_{i-1}$. Since it can never appear again, it is
deleted from $\cF$. (ii) Otherwise it is a ``new'' facet and is kept
in $\cF$.

\end{remark}

\section{Cutting cones by halfspaces}\label{dual_alg}

The primal algorithm of Normaliz builds a cone $C$ by starting from
$0$ and adding the generators $x_1,\dots,x_n$ successively. The
algorithm we want to discuss now (esentially due to Pottier
\cite{Pot}) builds the dual cone $C^*$ successively by staring from
$0$ and adding generators $\lambda_1,\dots,\lambda_s$. On the primal
side this amounts to cutting out the cone $C$ from $0^*=\RR^d$ by
successively intersecting the cone reached with the halfspace
$H_{\lambda_i}^+$, $i=1,\dots,s$, until one arrives at $C$.

If one wants to compute the Hilbert basis of $C$ via this
construction, then one has to understand how to obtain the Hilbert
basis of an intersection $D\cap H^+$ from that of the cone $D$.

Since we start from the full space $\RR^d$ and we cannot reach a
pointed cone before having cut it with at least $d$ halfspaces, we
use the general notion of a Hilbert basis as introduced in Section
\ref{affmon}. (In the following we do not assume that $C$ or $C^*$
is a full-dimensional cone.) Of course, in addition to the Hilbert
basis $B$ of $M=C\cap\ZZ^d$, we also need a description of the group
$\U(M)$ by a $\ZZ$-basis.

The halfspace $H^+$ is given by an integral linear form $\lambda$,
$H^+=H_\lambda^+$. In the following the superscript~$\vphantom{M}^+$
denotes intersection with $H_\lambda^+$, and the
superscript~$\vphantom{M}^-$ denotes intersection with
$H_\lambda^-$.

There are two cases that must be distinguished:
\begin{itemize}
\item[(a)] $\lambda$ vanishes on $\U(M)$; in this case $\U(M^+)=\U(M^-)=\U(M)$.
\item[(b)] $\lambda$ does not vanish on $\U(M)$; in this case
$\U(M^+)=\U(M^-)$ is a proper subgroup of $\U(M)$ such that $\rank
\U(M^+)=\rank \U(M)-1$. Moreover $\U(M)^+$ has a Hilbert basis
consisting of a single element $h$, and then $-h$ constitutes a
Hilbert basis of $\U(M)^-$.
\end{itemize}
Note that case (a) automatically applies if $D$ is pointed.

The following algorithm computes the Hilbert bases of $M^+$, $M^-$
and a basis of the group $\U(M^+)=\U(M^-)$, starting from a Hilbert
basis $B$ of $M$ and a basis of $\U(M)$.

\begin{itemize}
\item[(D1)] Compute a basis of $\U(M^+)=\U(M^-)=\Ker \lambda|\U(M)$. If
$\U(M^+)=\U(M)$, then we are in case (a). Otherwise $\rank
\U(M^+)=\rank \U(M)-1$ and we are in case (b).

\item[(D2)] In case (b) supplement the basis of $\U(M^+)$ to a basis
of $\U(M)$ by an element $h\in\U(M)^+$.

\item[(D3)] Set $B_0=B$.

\item[(D4)] In case (b) replace every element $x\in B_0^+$ by $x-ah$ where
$a=\lfloor \lambda(x)/\lambda(h)\rfloor$, and every element $x\in
B_0^-$ $x-a(-h)$, $a=\lfloor \lambda(-x)/\lambda(h)\rfloor$.

\item[(D5)] In case (b) replace $B_0$ by $B_0\cup\{h,-h\}$.

\item[(D6)] For $i>0$ set
$$
\tilde B_i=B_{i-1}\cup \bigl\{x+y:x,y\in B_{i-1},\
\lambda(x)>0,\lambda(y)<0, x+y\neq 0\bigr\}.
$$

\item[(D7)] Replace $\tilde B_i^+$ by its auto-reduction $B_i^+$ in $D^+$,
and $\tilde B_i^-$ by its auto-reduction $B_i^-$ in $D^-$, and let
$B_i=B_i^+\cup B_i^-$.

\item[(D8)] If $B_i=B_{i-1}$, then we are done, $B_{i-1}^+$ is a Hilbert basis
of $D^+$, and $B_{i-1}^-$ is a Hilbert basis of $D^-$.
\end{itemize}
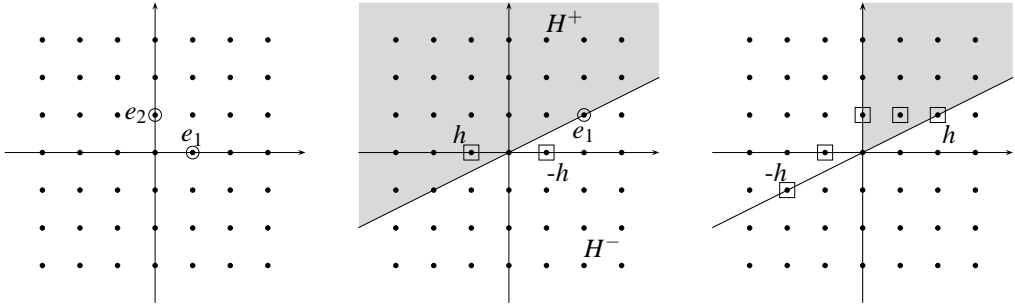
\begin{figure}
\footnotesize
$$
\psset{unit=0.5cm, linewidth=0.3pt}
\definecolor{veryverylight}{gray}{0.85}
\begin{pspicture}(-4,-4)(4,4)
\multirput(-3,-3)(1,0){7}{\multirput(0,0)(0,1){7}{\vertex}}
\psline{->}(0,-4)(0,4)
\psline{->}(-4,0)(4,0)
\rput(1,0.4){$e_1$}
\rput(-0.5,1){$e_2$}
\rput(0,1){\pscircle{0.18}}
\rput(1,0){\pscircle{0.18}}
\end{pspicture}
\qquad
\begin{pspicture}(-4,-4)(4,4)
\pspolygon[linecolor=white, linewidth=0pt, fillstyle=solid,
fillcolor=veryverylight](-4,-2)(4,2)(4,4)(-4,4)
\multirput(-3,-3)(1,0){7}{\multirput(0,0)(0,1){7}{\vertex}}
\psline{->}(0,-4)(0,4)
\psline{->}(-4,0)(4,0)
\psline[linewidth=0.4pt](-4,-2)(4,2)
\rput(1.5,3.5){$H^+$}
\rput(2.5,-2.5){$H^-$}
\rput(1.3,-0.5){-$h$}
\rput(-1.3,0.5){$h$}
\rput(2,0.5){$e_1$}
\rput(2,1){\pscircle{0.18}}
  \rput(-1,0){\pspolygon(-.2,-.2)(.2,-.2)(.2,.2)(-.2,.2)}
  \rput(1,0){\pspolygon(-.2,-.2)(.2,-.2)(.2,.2)(-.2,.2)}
\end{pspicture}
\qquad
\begin{pspicture}(-4,-4)(4,4)
\pspolygon[linecolor=white, linewidth=0pt, fillstyle=solid,
fillcolor=veryverylight](0,0)(4,2)(4,4)(0,4)
\multirput(-3,-3)(1,0){7}{\multirput(0,0)(0,1){7}{\vertex}}
\psline[linewidth=0.4pt]{->}(0,-4)(0,4)
\psline{->}(-4,0)(4,0)
\psline[linewidth=0.4pt](-4,-2)(4,2)
\rput(2.3,0.5){$h$}
\rput(-2.3,-0.5){-$h$}
\rput(0,1){\pspolygon(-.2,-.2)(.2,-.2)(.2,.2)(-.2,.2)}
\rput(1,1){\pspolygon(-.2,-.2)(.2,-.2)(.2,.2)(-.2,.2)}
\rput(2,1){\pspolygon(-.2,-.2)(.2,-.2)(.2,.2)(-.2,.2)}
\rput(-1,0){\pspolygon(-.2,-.2)(.2,-.2)(.2,.2)(-.2,.2)}
\rput(-2,-1){\pspolygon(-.2,-.2)(.2,-.2)(.2,.2)(-.2,.2)}
\end{pspicture}
$$
\caption{Successive cuts with halfspaces}\label{success}
\end{figure}
The construction is illustrated by Figure \ref{success}; base
elements of the unit groups have been marked by a circle, Hilbert
basis elements by a square.

We have to prove the claim contained in (D8), and we state it as a
lemma.

\begin{lemma}\label{dualHB}
There exists an $i\ge 1$ such that $B_i=B_{i-1}$, and in this case
$B_i^+=\Hilb(M^+)$, $B_i^-=\Hilb(M^-)$.
\end{lemma}

\begin{proof}
Let $B_\infty=\bigcup_{i=0}^\infty B_i$. We will show that
$B_\infty^+$ generates $M^+$ modulo $\U(M^+)$ and $B_\infty^-$ does
the same for $M^-$. In other words, we claim that for every $x\in
M^+$ there exist $u_1,\dots,u_r\in B_\infty^+$ such that
\begin{equation}\label{geneq}
x-(u_1+\dots+u_r)\in \U(M^+),
\end{equation}
and the corresponding statement holds for $B_\infty^-$ and $M^-$.

Suppose this claim has been proved. Then $B_\infty^+$ contains a
Hilbert basis of $M^+$ since it is a system of generators modulo
$\U(M^+)$. Since the Hilbert basis contains only irreducible
elements (an irreducible element will pass step (D7) above) and is
finite, there must be an $i$ for which $B_{i-1}^+$ contains the
Hilbert basis (in fact, equals it). Then $B_i^+=B_{i-1}^+$.
Increasing $i$ if necessary, we also have $B_i^-=B_{i-1}^-$, and
then $B_i=B_{i-1}$. Conversely, if $B_i=B_{i-1}$, then
$B_{i-1}^+=\Hilb(M^+)$ and $B_{i-1}^-=\Hilb(M^-)$.

We (have to) prove the crucial claim simultaneously for $M^+$ and
$M^-$, considering the more complicated case (b). The proof for case
(a) is obtained if one omits all those arguments that refer to $h$.

We use induction on $\tdeg x$, the total degree with respect to $M$.
We can assume that $x\in M^+$ since the argument for $x\in M^-$ is
analogous. If $\tdeg x=0$, we have $x\in\U(M)$. But then $x-ah\in
\U(M^+)$, since $h$ is a Hilbert basis of $\U(M)^+$ modulo its group
$\U(M^+)$ of invertible elements. Moreover, $h\in B_\infty^+$ by
construction.

Suppose that $\tdeg x>0$, and note that $x$ has a representation
\begin{equation}\label{repr1}
x\equiv (u_1+\dots+u_r)+(v_1+\dots+v_s)+(w_1+\dots+w_t)\mod \U(M^+)
\end{equation}
modulo $\U(M^+)$ in which $u_j, v_k, w_l \in B_\infty $ and
$\lambda(u_j)>0$, $\lambda(v_k)=0$ and $\lambda(w_l)<0$. In fact,
since $B_\infty$ contains a Hilbert basis of $M$, we can find such a
representation modulo $\U(M)$, and adding $h$ or $-h$ sufficiently
often, we end up in $\U(M^+)$.

Among all the representations \eqref{repr1} we choose an
\emph{optimal} one, namely one for which $ \lambda(u_1+\dots+u_r)$
is minimal. If we can show that $t=0$ for this choice, then we are
done. Note that only one of $h$ or $-h$ can appear in an optimal
representation; otherwise canceling $h$ against $-h$ would improve
it.

Clearly, if $t>0$, then $r>0$ as well, since otherwise
$\lambda(x)\ge 0$ is impossible. Consider the representation
\begin{equation}\label{repr2}
x\equiv(u_1+w_1)+(u_2+\dots+u_r)+(v_1+\dots+v_s)+(w_2+\dots+w_t)\mod
\U(M^+)
\end{equation}
modulo $\U(M^+)$.

If $u_1+w_1$ belongs to $B_\infty$ we are done, since
$\lambda(u_1+w_1)<\lambda(u_1)$, regardless of the sign of
$\lambda(u_1+w_1)$.

Otherwise $u_1+w_1$ is reducible in step (D7). Assume $u_1+w_1\in
M^+$ (an analogous argument can be given if $u_1+w_1\in M^-$). Then
there exists $y\in B_\infty^+$ such that $(u_1+w_1)-y\in M^+$. Note
that $y=h$ is impossible: by construction, all elements $z$ of
$B_\infty$ different from $h$ and $-h$ have
$|\lambda(z)|<\lambda(h)$, so
$\lambda(u_1+w_1)<\lambda(u_1)<\lambda(h)$ and $(u_1+w_1)-h\notin
M^+$. But all elements of $B_\infty$ different from $h$ and $-h$ do
not belong to $\U(M)$ and therefore have positive total degree in
$M$. Then $\tdeg(u_1+w_1-y)<\tdeg(u_1+w_1)$. Thus we can apply
induction to $(u_1+w_1)-y$, representing it modulo $\U(M^+)=\U(M^-)$
by elements from $B_\infty^+$. Since $y\in B_\infty^+$, we obtain a
representation for $u_1+w_1$. Substituting this representation into
\eqref{repr2} again yields an improvement. This is a contradiction
to the choice of \eqref{repr1}, and we are done.
\end{proof}

While the description of the algorithm given above is very close to
the implementation in Normaliz, we would like to mention some
further details.

\begin{remark}
(a) It is clear that in the formation of $\tilde B_i$ in step (D6)
one should avoid the sums $x+y$ that have already been formed in an
earlier ``generation''.

(b) When a sum $x+y$ has been formed, it is immediately tested
against reducibility by $B_{i-1}^+$ or $B_{i-1}^-$, respectively.
The elements that survive are collected, and the remaining reduction
steps are applied after this collection.

This ``generation driven'' procedure has the advantage that we can
apply the rather efficient reduction strategy of Section
\ref{reduction}.

(c) Sometimes an element is reduced by one that is created in a
later generation. However, this happens rarely.

(d) When a sum $z=x+y$ is formed and belongs to $\tilde B_i^+$, then
we store $\lambda(x)$ with $z$ (analogously if $x+y\in M^-$).
Suppose that $z$ survives the reduction. Then it is not necessary to
form sums $z+w$ with $\lambda(w)<-\lambda(x)$ since we would have
$x+w,z+w\in M^-$, and $x+w$ clearly reduces $z+w$. This trick
diminishes the number of sums to be formed in higher generations
considerably, but does not help in the formation of $B_1$, which is
usually (but not always) the most time consuming generation.

(e) Normaliz uses a heuristic rule to determine the order in which
the hyperplanes are inserted into the algorithm. It is evidently
favorable to keep the sets $B_i$ small as long as possible.

(f) In a future version of Normaliz we will also try a hybrid
approach in which the algorithm for the local Hilbert bases is
chosen dynamically.
\end{remark}

We conclude the article by a comparison of the primal algorithm of
Normaliz and the algorithm of this section.

\begin{remark}
(a) Normaliz has two input modes in which the cone $C$ is specified
by inequalities (and equations). In these cases the user can choose
whether to apply the primal algorithm (first computing a system of
generators of $C$) or the dual algorithm described in this section.
It is not easy to decide which of the two algorithms will perform
better for a given $C$. The bottleneck of the primal algorithm is
certainly the computation of a full triangulation (if it is done).
The size of the triangulation is mainly determined by the number of
support hyperplanes of the subcones of $C$ through which the
computation passes. However, if it can be found, the (partial)
triangulation itself carries a large amount of information, and the
subsequent steps profit from it.

We illustrate these performance of the primal algorithm by two
examples, one for which it is very fast, and another one that it
cannot solve.

(i) The example \texttt{small} from the Normaliz distribution is
defined by a $5$-dimensional lattice polytope with $190$ vertices,
$32$ support hyperplanes, and $34,591$ lattice points. Its
normalized volume is $2,276,921$, and the triangulation contains
$1593$ simplicial cones (if computed in the mode ``normal''). About
$230,000$ vectors survive the local reduction, and are sent into
global reduction, leading to the Hilbert basis with $34,591$
vectors. (The number of candidates for global reduction caries
considerably with the triangulation; those derived from shellings
seem to behave worse in this respect.) Run time with Normaliz 2.2
(the currently public version) on our SUN Fire X4450 is $8$ sec
(and 19 sec if the $h$-vector is computed).

(ii) The example \texttt{5x5} from the Normaliz distribution
describes the cone of $5\times5$ ``magic squares'' \cite{ADH}, i.e.,
$5\times5$ matrices with nonnegative entries and constant row,
column and diagonal sums. The cone of dimension $15$ has $1940$
extreme rays and $25$ support hyperplanes. The subcone generated by
the first $57$ extreme rays (in the order Normaliz finds them) has
already $30,290$ support hyperplanes. After $104$ extreme rays we
reached $56,347$ support hyperplanes (and we stopped the program).

(b) The main obstruction in the application of the dual algorithm is
the potentially extremely large number of vectors it has to
generate. Even if the Hilbert basis of the final cone $C$ is small,
the Hilbert bases of the overcones of $C$ through which the
algorithm passes may be extremely large, or one has to compute with
medium size Hilbert bases in many successive overcones. Some data on
the behavior of the algorithm on the two examples from (a)---now (i)
is hard, and (ii) is easy:

(i) For the lattice polytope the dual algorithm (staring from the
support hyperplanes) needs $3540$ sec. One of the intermediate
Hilbert bases has cardinality $145,098$. Therefore the number of
elements of $\tilde B_1$ at the insertion of the next hyperplane can
safely be estimated by $10^9$.

(ii) After $20$ hyperplanes have been inserted, the size of the
Hilbert basis of the cone reached is $228$, and the values for the
subsequent cones are $979$, $1836$, $2810$, $3247$, and finally
$4828$. Computation time is 2 sec.

(c) If the Hilbert basis of $C=\RR_+x_1+\dots+x_n$ is to be
computed, the primal algorithm builds an ascending chain
$$
0=C_0\subset C_1\subset\dots\subset C_n=C,
$$
with a corresponding descreasing chain
$$
(\RR^d)^*=C_0^*\supset C_1^*\supset\dots\supset C_n^*=C^*.
$$
For the cone $D=\bigcap_{j=1}^s H_{\lambda_i}^+$ the dual algorithm
proceeds in exactly the opposite way, building an increasing chain
$$
0=D_0^*\subset D_1^*\subset\dots\subset D_s^*=D^*,
$$
of dual cones, with a corresponding decreasing chain
$$
\RR^d=D_0\supset D_1\supset\dots\supset D_s=D.
$$
In both cases, the complexity is determined by the
\emph{decreasing} chains of overcones of $C^*$ and $D$,
respectively. These overcones are hard to control only by the
internal data of $C$ or $D^*$.

Of course, if $n=1940$ and $s=25$ as in example (ii), then the
choice is easy, but example (i) with $n=190$ and $s=32$
illustrates that the sole comparison of $n$ and $s$ does in
general not help to pick the better algorithm.
\end{remark}

We conclude by presenting the following table which contains
experimental test data we have obtained for computing the
Hilbert basis with Normaliz version 2.2, as well as data
obtained from our tests with 4ti2 version 1.3.2.  The system
4ti2 \cite{He}, \cite{4ti2} contains a somewhat different
implementation of the dual algorithm. The table shows that our
version is certainly comparable in performance.\smallskip
\begin{center}
\begin{tabular}{|r|r|r|r|r|r|r|r|r|}
\hline
 name & input&dim&\#gen&\#supp&\#HB&$t$ primal&$t$ dual&$t$ 4ti2 \\
\hline
cut.in & supp & 5 & 83 & 25 & 4,398 & 0.15 &  3.4 & 620 \\
\hline
small.in & gen & 6 & 190 & 32 & 34,591 & 8 &  3,540 & 3,230 \\
\hline
medium.in & gen & 17 & 40 & 3,753 & 217 & 11 &  $\infty$ & $\infty$ \\
\hline
4x4.in & equ & 8 & 20 & 34 & 20 & 0.002 & 0.001 & 0.01 \\
\hline
5x5.in & equ & 15 & 1,940 & 47 & 4,828 & $\infty$ & 2 & 2.4 \\
\hline
6x6.in & equ & 24 & 97,548 & 62 & 522,347 & $\infty$ & 87,600 & 345,600 \\
\hline
\end{tabular}
\end{center}\smallskip

The first column refers to the name of the input file in the
Normaliz distribution. The second column describes the type of
input, namely \textit{gen}erators, \textit{supp}ort hyperplanes, or
system of \textit{equ}ations. In the latter case the cone is the
intersection of the solution space with the nonnegative orthant. The
third column contains the dimension of the cone, and the following
three list the number of its generators, support hyperplanes and
Hilbert basis elements. The last three columns contain computation
times for Normaliz primal, Normaliz dual and 4ti2 (measured in
seconds). For the application of Normaliz dual or 4ti2 to input of
type ``gen'', we first computed the dual cone separately (or
extracted it from the output of the primal algorithm). Normaliz
primal, when applied to any type of input, does the necessary
dualization itself.

We thank Christof Söger for measuring the computation times.

\end{document}